\documentclass[10pt,a4paper]{article}

\usepackage{latexsym,amsfonts,amsmath,amssymb,mathrsfs,url,color,amsthm}
\usepackage{graphicx,cite}
\usepackage{siunitx}

\usepackage{natbib}
 \bibpunct[, ]{(}{)}{,}{a}{}{,}%
 \def\BIBand{and}%

\def\EMAIL#1{#1}

\newtheorem{theorem}{Theorem}
\newtheorem{lemma}{Lemma}
\newtheorem{assumption}{Assumption}

\newtheorem{remark}{Remark}

\usepackage[dvipdfm,
linkbordercolor={1 0 0},
citebordercolor={0 1 0},
urlbordercolor={0 1 1}]{hyperref}

\newcommand{\fracs}[2]{{\textstyle \frac{#1}{#2}}}

\def \PP  {{\mathbb{P}}}
\def \EE  {{\mathbb{E}}}
\def \VV  {{\mathbb{V}}}
\def \RR  {{\mathbb{R}}}

\def \ofd  {{\overline{f_d}}}

\def \eps {{\varepsilon}}
\def \Eps {{\epsilon}}

\def \one {{\bf{1}}}

\def \XN  {{\overline{X}_{\!N}}}

\usepackage{color}

\allowdisplaybreaks

\begin{document}

\title{Decision-making under uncertainty:\\ using MLMC for efficient estimation of EVPPI}

\author{Michael B.~Giles\thanks{Mathematical Institute, University of Oxford, Oxford, United Kingdom, OX2 6GG, \EMAIL{mike.giles@maths.ox.ac.uk}}, 
Takashi Goda\thanks{School of Engineering, University of Tokyo, Tokyo 113-8656, Japan, \EMAIL{goda@frcer.t.u-tokyo.ac.jp}}}

\date{\today}

\maketitle

\begin{abstract}
In this paper we develop a very efficient approach to the Monte Carlo estimation
of the expected value of partial perfect information (EVPPI) that measures 
the average benefit of knowing the value of a subset of uncertain parameters 
involved in a decision model.  The calculation of EVPPI is inherently a nested 
expectation problem, with an outer expectation with respect to one random 
variable $X$ and an inner conditional expectation with respect to 
the other random variable $Y$. We tackle this problem by using a Multilevel Monte 
Carlo (MLMC) method \citep{giles08} in which the number of inner samples for $Y$ 
increases geometrically with level, so that the accuracy of estimating the inner 
conditional expectation improves and the cost also increases with level.  We construct 
an antithetic MLMC estimator and provide sufficient assumptions on a decision model 
under which the antithetic property of the estimator is well exploited, and consequently 
a root-mean-square accuracy of $\eps$ can be achieved at a cost of $O(\eps^{-2})$.  
Numerical results confirm the considerable computational savings compared to the 
standard, nested Monte Carlo method for some simple testcases and a more realistic 
medical application.
\end{abstract}

\section{Introduction}

The motivating applications for this research come from two
apparently quite different fields, the funding of medical research
and the exploration and exploitation of oil and gas reservoirs.  
The common element in both cases is decision making under a large 
degree of uncertainty.

In the medical case \citep{alc04,bkoc07}
let $X$ and $Y$ represent independent random
variables representing the uncertainty in the effectiveness of 
different medical treatments. In the absence of any knowledge 
of $X$ or $Y$, then given a finite set of possible treatments $D$, 
the optimal choice $d_{opt}$ is the one which maximises 
$\EE\left[ f_d (X,Y) \right]$ where $f_d(X,Y)$ represents some 
measure of the patient outcome, such as QALY's (quality-adjusted 
life-year), measured on a monetary scale with a larger value being better.
Thus, with no knowledge, the optimal outcome on average is
\[
\max_{d\in D}\ \EE\left[ f_d (X,Y) \right].
\]
On the other hand, given perfect information on $X$ and $Y$, 
through carrying out some new medical research, the best 
treatment choice maximises $f_d (X,Y)$, giving the overall 
average outcome 
\[
\EE\left[\, \max_{d\in D}\, f_d (X,Y) \right].
\]
In the intermediate situation, if $X$ is known but not $Y$, 
then the best treatment has average outcome value
\[
\EE \left[ \max_d \EE\left[ f_d (X,Y) \,|\, X\right] \right].
\]

EVPI, the expected value of perfect information, 
is the difference
\[
\mbox{EVPI} = 
\EE \left[ \max_d f_d (X,Y) \right]
- \max_d\, \EE[ f_d (X,Y) ],
\]
and EVPPI, the expected value of partial perfect information, 
is the difference
\[
\mbox{EVPPI} = 
\EE \left[ \max_d\, \EE\left[ f_d (X,Y) \,|\, X \right] \right]
- \max_d\, \EE[ f_d (X,Y) ].
\]
EVPPI represents the benefit, on average, of knowing the value of 
$X$.  If the value of $X$ represents the information arising from 
a proposed piece of medical research, then one can compare the cost 
of the research to the benefits which arise from the information 
obtained. 

In the oil and gas reservoir scenario \citep{bbl09,ngts16}, there are also decisions to 
be made, such as whether or not to drill additional exploratory wells.
There is huge uncertainty in various aspects of an oil reservoir, 
its dimensions, the oil and gas reserves it contains, the rock porosity,
etc.  An additional well will yield information which will reduce the
uncertainty and increase, on average, the amount of oil and gas which
will eventually be extracted. However, there is an additional cost in 
drilling one more well, and the EVPPI will help determine whether or not 
it is worth it.

The calculation of EVPPI is a nested expectation problem, with an outer 
expectation over $X$ and an inner conditional expectation over $Y$.
In this paper, we choose to focus on the estimation of the difference
\[
\mbox{EVPI} - \mbox{EVPPI} =
\EE \left[ \max_d\, f_d (X,Y) \right]
- \EE \left[ \max_d\, \EE\left[ f_d (X,Y) \,|\, X \right] \right].
\]
EVPI can be estimated directly using standard Monte Carlo methods 
with independent samples of $(X,Y)$
\[ \frac{1}{N}\sum_{n=1}^{N}\max_d f_d(X^{(n)},Y^{(n)}) - \max_d \frac{1}{N}\sum_{n=1}^{N}f_d(X^{(n)},Y^{(n)}). \]
Assuming each computation $f_d (X,Y)$ can be performed with unit 
cost, EVPI can be estimated with root-mean-square accuracy $\eps$
by using $N=O(\eps^{-2})$ samples $(X^{(n)}, Y^{(n)})$ at a total cost 
which is $O(\eps^{-2})$.
On the other hand, estimating the difference $\mbox{EVPI} - \mbox{EVPPI}$ 
using standard, nested Monte Carlo methods requires 
$N$ outer samples of $X$ and $M$ inner samples of $Y$, giving
\[ \frac{1}{N}\sum_{n=1}^{N}\left[ \frac{1}{M}\sum_{m=1}^{M}\max_d f_d(X^{(n)},Y^{(n,m)}) 
 - \max_d \frac{1}{M}\sum_{m=1}^{M}f_d(X^{(n)},Y^{(n,m)})\right]. \]
As shown in the next section, in order to estimate $\mbox{EVPI} - \mbox{EVPPI}$ 
with root-mean-square accuracy $\eps$ by this estimator, 
we need $N=O(\eps^{-2})$ and $M=O(\eps^{-1/\alpha})$ samples
for outer and inner expectations, respectively. 
Here $\alpha>0$ denotes the order of convergence of the bias 
and is typically between $1/2$ and $1$.
Therefore, the computational complexity will be at least $O(\eps^{-3})$, 
and increase up to $O(\eps^{-4})$ in the worst case.

The aim of this paper is to develop an efficient approach to this nested 
expectation problem, i.e., the estimation of $\mbox{EVPI} - \mbox{EVPPI}$, 
by using a Multilevel Monte Carlo (MLMC) method \citep{giles15}.
MLMC estimators have been used previously for nested expectations of the 
slightly different form $\EE[f(\EE[Y|X])]$ by \cite{hajiali12} and 
\cite{giles15} for cases in which $f$ is twice-differentiable, and by 
\cite{bhr15} for a case in which $f$ is continuous and piecewise linear.  
Current research \citep{gh18} is also looking at the case in which $f$ 
is a discontinuous indicator (Heaviside) function.

Building on this prior MLMC research, we 
introduce an antithetic MLMC estimator for $\mbox{EVPI} - \mbox{EVPPI}$
in the next section, and then in Section~\ref{sec:variance}, we provide sufficient assumptions 
on $f_d$'s such that the antithetic property of the estimator is well exploited, 
and by building upon the basic MLMC theorem (Theorem~\ref{thm:MLMC}), 
the estimator is proven to achieve the optimal computational complexity 
$O(\eps^{-2})$ (Theorem~\ref{thm:1}).
Numerical experiments in Section~\ref{sec:numerics} confirm the importance of the assumptions
made in our theoretical analysis, and also the considerable computational savings
compared to the standard, nested Monte Carlo method
not only for some simple testcases 
but also for a more realistic medical application.

\section{MLMC method}

\subsection{Basic MLMC theory}
The MLMC method was introduced by 
\citet{heinrich01} for parametric integration, 
and by \citet{giles08} for the estimation of the expectations 
arising from SDEs.
It was subsequently extended to SPDEs \citep[e.g.][]{cgst11}, 
stochastic reaction networks \citep{ah12}, 
and nested simulation \citep{hajiali12,bhr15}.
For an extensive review of MLMC methods, see the review 
by \citet{giles15}.

Here we give a brief overview of the MLMC method.
The problem we are interested in is to estimate $\EE[P]$ efficiently
for a random output variable $P$ which cannot be sampled exactly.
Given a sequence of random variables $P_0,P_1,\ldots$ which
approximate $P$ with increasing accuracy but also with increasing cost, 
we have the elementary telescoping summation
\begin{equation}
\EE[P_L] = \EE[P_0] + \sum_{\ell=1}^L \EE[P_\ell-P_{\ell-1}].
\label{eq:MLMC}
\end{equation}
The key idea behind the MLMC method is 
to independently estimate each of the 
quantities on the r.h.s.~of (\ref{eq:MLMC}) instead
of directly estimating the l.h.s., which is the standard Monte Carlo approach.
For the same underlying stochastic sample, $P_\ell$ and $P_{\ell-1}$ could be 
well correlated each other, and the variance of the correction $P_\ell-P_{\ell-1}$ 
is expected to get smaller as the level $\ell$ increases. Thus, in order to estimate 
each of the quantities on the r.h.s.~of (\ref{eq:MLMC}) with the same accuracy, 
the necessary number of samples for the finest levels becomes much smaller than
that for the coarsest levels, resulting in a significant reduction of 
the total computational cost as compared to the standard Monte Carlo method. 
This observation leads to the following theorem \citep{giles15}:

\begin{theorem}
\label{thm:MLMC}
Let $P$ denote a random variable, and let $P_\ell$ denote the 
corresponding level $\ell$ numerical approximation.
If there exist independent random variables $Z_\ell$ 
with expected cost $C_\ell$ and variance 
$V_\ell$, and positive constants 
$\alpha, \beta, \gamma, c_1, c_2, c_3$ such that 
$\alpha \geq \fracs{1}{2}\min(\beta,\gamma)$ and
\begin{itemize}
\item[i)] ~
$\displaystyle
\left| \EE[P_\ell - P] \right| \leq c_1 2^{-\alpha \ell}
$
\item[ii)] ~
$\displaystyle
\EE[Z_\ell] = \left\{ \begin{array}{ll}
\EE[P_0],                     & \ell=0 \\
\EE[P_\ell - P_{\ell-1}], & \ell>0
\end{array}\right.
$
\item[iii)] ~
$\displaystyle
V_\ell \leq c_2 2^{-\beta \ell}
$
\item[iv)] ~
$\displaystyle
C_\ell \leq c_3 2^{\gamma \ell},
$
\end{itemize}
then there exists a positive constant $c_4$ such that for any 
$\eps \!<\! e^{-1}$
there are values $L$ and $N_\ell$ for which the multilevel estimator
\[
\hat{Z} = \sum_{\ell=0}^L \hat{Z}_\ell \quad \text{with} \quad \hat{Z}_\ell = \frac{1}{N_\ell}\sum_{n=1}^{N_\ell}Z_\ell^{(n)},
\]
has a mean-square-error with bound
\[
MSE \equiv \EE\left[ \left(\hat{Z} - \EE[P]\right)^2\right] < \eps^2
\]
with a computational complexity $C$ with bound
\[
\EE[C] \leq \left\{\begin{array}{ll}
c_4 \eps^{-2}              ,    & \beta>\gamma, \\
c_4 \eps^{-2} (\log \eps)^2,    & \beta=\gamma, \\
c_4 \eps^{-2-(\gamma-\beta)/\alpha}, & \beta<\gamma.
\end{array}\right.
\]
\end{theorem}

\begin{remark}
\label{rem:alpha_beta}
In the case where the condition $V_\ell \leq c_2 2^{-\beta \ell}$ can be replaced by
$\EE[Z_\ell^2]\leq c_2 2^{-\beta \ell}$, H{\"o}lder's inequality gives
\[ \EE[Z_\ell] \leq \left( \EE[Z_\ell^2]\right)^{1/2}\leq \sqrt{c_2} 2^{-\beta \ell/2}. \]
Using the triangle inequality, we obtain
\begin{align*}
\left| \EE[P - P_L] \right| &  = \left| \sum_{\ell\geq L}\EE[P_{\ell+1} - P_\ell] \right| 
= \left| \sum_{\ell\geq L}\EE[Z_\ell] \right| \\
& \leq \sum_{\ell\geq L}\left| \EE[Z_\ell] \right| 
\leq \frac{\sqrt{c_2}}{1-2^{-\beta/2}}2^{-\beta L/2}.
\end{align*}
Compared this bound to the condition $\left| \EE[P_\ell - P] \right| \leq c_1 2^{-\alpha \ell}$,
we have $\alpha\geq \beta/2$ and so the assumption 
$\alpha \geq \fracs{1}{2}\min(\beta,\gamma)$ is simplified into $\alpha \geq \gamma/2$.
\end{remark}

As far as possible, we try to develop MLMC estimators 
which are in the first regime, with $\beta\!>\!\gamma$, 
so that the total cost is $O(\eps^{-2})$.  This corresponds to $O(\eps^{-2})$ samples 
each with an average $O(1)$ cost, and it means that most of the 
computational cost is incurred on the coarsest levels. When the 
application is in this regime, \citet{rg15} have a technique in which
they randomise the selection of the level $\ell$ to obtain a method 
which is unbiased but has a finite variance and average cost per sample.

Nevertheless, in any regime, 
Theorem~\ref{thm:MLMC} compares favourably with the complexity bound
for the standard Monte Carlo method which directly estimates 
the l.h.s.~of (\ref{eq:MLMC}) based on $N$ Monte Carlo samples
of $P_L$ for a fixed $L$:
\[
\hat{Z}' = \frac{1}{N}\sum_{n=1}^{N}P_L^{(n)}.
\]
In addition to the conditions given in Theorem~\ref{thm:MLMC},
assume $V:=\sup_{\ell}\VV[P_\ell]<\infty$.
For a given accuracy $\eps$, let us choose 
$N=\lceil 2V\eps^{-2}\rceil$ and
$L=\lceil \log_2(\sqrt{2}c_1\eps^{-1})/\alpha\rceil$,
so that the variance and the bias of the estimator are bounded simultaneously:
\[
\VV[\hat{Z}'] = \frac{\VV[P_L]}{N} \leq \frac{\VV[P_L]}{2V}\eps^2 \leq \frac{\eps^2}{2},
\]
and 
\[
\left(\EE[P-P_L]\right)^2 \leq \frac{c_1^2}{2^{2\alpha L}} \leq \frac{\eps^2}{2},
\]
which ensures the mean-square-error bound of $\hat{Z}'$
\[
\EE\left[(\hat{Z}'-\EE[P])^2\right] = \VV[\hat{Z}'] +\left(\EE[P-P_L]\right)^2\leq \eps^2.
\]
Then there exists a positive constant $c_5$ such that the expected cost of $\hat{Z}'$ is bounded by
\[
NC_L\leq \left(2V\eps^{-2}+1\right)c_3 2^{\gamma L} 
\leq \left(2V\eps^{-2}+1\right)c_3\left( \sqrt{2}c_12^\alpha\eps^{-1}\right)^{\gamma/\alpha} 
\leq c_5\eps^{-2-\gamma/\alpha}.
\]
In general, it seems hard to improve the exponent $2+\gamma/\alpha$ of $\eps^{-1}$.
Therefore, the multilevel estimator always has an asymptotically better complexity bound than
the standard Monte Carlo estimator.

\subsection{MLMC estimator for $\mbox{EVPPI}$}
In view of the previous subsection, for the estimation of the difference $\mbox{EVPI} - \mbox{EVPPI}$ 
let us define a random output variable $P$ by
\[ P = \EE\left[ \max_d f_d(X,Y)\, |\, X\right] - \max_d \EE\left[ f_d(X,Y)\, |\, X\right] \]
with the underlying stochastic variable $X$. Obviously $P$ is nothing but 
the inner conditional expectation of $\mbox{EVPI} - \mbox{EVPPI}$, and the problem
we tackle in this paper is rephrased into an efficient estimation of $\EE[P]$.
A sequence of random variables $P_0,P_1,\ldots$ is defined by
\[ P_\ell 
 = \frac{1}{2^\ell}\sum_{i=1}^{2^\ell}\max_d f_d(X,Y^{(i)}) 
 - \max_d \frac{1}{2^\ell}\sum_{i=1}^{2^\ell}f_d(X,Y^{(i)})
 =: \overline{\max_d\, f_d}^\ell - \max_d\, \ofd^\ell
\]
where $\overline{\max_d\, f_d}^\ell$ and $\ofd^\ell$ represent averages over $2^\ell$ 
independent values of $Y^{(i)}$ for a randomly chosen $X$, respectively.
That is to say, $P_\ell$ simply denotes the standard Monte Carlo estimator 
based on $2^\ell$ samples for the inner conditional expectation of $\mbox{EVPI} - \mbox{EVPPI}$, 
so that the sequence $P_0,P_1,\ldots$ approximate $P$ 
with increasing accuracy but also with increasing cost.
Namely we have
\[
\mbox{EVPI} - \mbox{EVPPI} = \EE[P] = 
\lim_{\ell\rightarrow\infty} \EE[P_\ell].
\]

As discussed above, in order to achieve a given accuracy $\eps$, 
the standard, nested Monte Carlo method chooses 
$N=O(\eps^{-2})$ and $M=O(2^L)=O(\eps^{-1/\alpha})$,
and so the computational complexity is $O(\eps^{-2-1/\alpha})$.
Using the MLMC method, this can be reduced significantly.
Following the ideas of \citet{hajiali12,bhr15,giles15}, we use 
an ``antithetic'' MLMC estimator 
\[
\hat{Z} = \sum_{\ell=1}^L \hat{Z}_\ell \quad \text{with} \quad \hat{Z}_\ell = \frac{1}{N_\ell}\sum_{n=1}^{N_\ell}Z_\ell^{(n)},
\]
in which 
\begin{align*}
Z_\ell & = \fracs{1}{2}\left(  \max_d \frac{1}{2^{\ell-1}}\sum_{i=1}^{2^{\ell-1}}f_d(X,Y^{(i)})
 +  \max_d \frac{1}{2^{\ell-1}}\sum_{i=2^{\ell-1}+1}^{2^\ell}f_d(X,Y^{(i)}) \right) \\
& \qquad - \max_d \frac{1}{2^\ell}\sum_{i=1}^{2^\ell}f_d(X,Y^{(i)}) \\
& =: \fracs{1}{2} \left( \max_d \ofd^{(a)} + \max_d \ofd^{(b)} \right) -\max_d \ofd
\end{align*}
where, for a randomly chosen $X$,
\begin{itemize}
\item
$\ofd^{(a)}$ is an average of $f_d(X,Y)$
over $2^{\ell-1}$ independent samples for $Y$;
\item
$\ofd^{(b)}$ is an average over a second independent 
set of $2^{\ell-1}$ samples;
\item
$\ofd$ is an average over the combined set of $2^{\ell}$
inner samples.
\end{itemize}
It is straightforward to see that 
$\gamma=1$ and $\EE[Z_\ell] = \EE[P_\ell \!-\! P_{\ell-1}]$ for $\ell>0$.
Here we consider $Z_0=P_0 \equiv 0$, 
so that the sum of the multilevel estimator over $\ell$ starts from $\ell=1$.

Note that we have the antithetic property $\fracs{1}{2}(\ofd^{(a)}+\ofd^{(b)}) - \ofd = 0$,
and therefore $Z_\ell\!=\!0$ if the same decision $d$ maximises each 
of the terms in its definition.  This is the key advantage of the 
antithetic estimator, compared to the alternative $\ofd^{(a)}-\ofd$.

\begin{remark}
It is straightforward to extend the antithetic MLMC approach to estimate $\mbox{EVPI}$.
The difference is that with $\mbox{EVPI}$ all of the underlying random variables $X$ and $Y$ 
are inner variables; non are outer variables leading to a conditional expectation.
Such an MLMC estimator for the maximum of an unconditional expectation 
has been introduced by \citet{bg15}. As discussed in the introduction, however, 
$\mbox{EVPI}$ can be estimated with $O(\eps^2)$ complexity 
by using standard Monte Carlo methods already, so that the benefit is 
that one could use a randomisation technique by \citet{rg15} to obtain 
an unbiased estimator, which might be marginal in the current setting.
\end{remark}

\section{MLMC variance analysis}\label{sec:variance}

We first show that the MLMC estimator achieves the nearly optimal complexity of $O(\eps^{-2} (\log \eps)^2)$ 
under a quite mild assumption.

\begin{theorem}
\label{thm:nearly-optimal}
If $\EE \left[ \VV [ f_d(X,Y) \,|\, X ]\right] $ is finite for all $d$, 
\[ \VV \left[ Z_\ell \right] \leq \EE \left[ \left| Z_\ell \right|^2 \right] \leq \frac{6|D|}{2^\ell}\sum_{d}\EE \left[ \VV [ f_d(X,Y) \,|\, X ]\right] .\]
\end{theorem}

\begin{proof}
For any two $|D|$-dimensional vectors with components $a_d, b_d$,
\[
\left| \max_d a_d - \max_d b_d \right|
\leq \max_d |a_d - b_d |
\leq \sum_d |a_d - b_d |.
\]
Hence, by defining $F_d(X) = \EE \left[f_d(X,Y) \,|\, X\right]$, we obtain
\begin{align}
\left| Z_\ell \right|  &= \left|\fracs{1}{2} (\max_d \ofd^{(a)} + \max_d \ofd^{(b)}) - \max_d \ofd \right| \nonumber \\
&=
\left| \fracs{1}{2} (\max_d \ofd^{(a)}-\max_d F_d)
     + \fracs{1}{2} (\max_d \ofd^{(b)}-\max_d F_d)
     -  (\max_d \ofd - \max_d F_d) \right| \nonumber \\
&\leq
\sum_d \left(
  \fracs{1}{2} |\ofd^{(a)}- F_d| 
+ \fracs{1}{2} |\ofd^{(b)}- F_d| 
+   |\ofd - F_d|
\right),
\label{eq:max_to_sum}
\end{align}
and therefore, by Jensen's inequality, 
\begin{align*}
\left| Z_\ell \right|^2  & = \left( \fracs{1}{2} (\max_d \ofd^{(a)} + \max_d \ofd^{(b)}) - \max_d \ofd \right)^{2} \\
&\leq
|D| \sum_d \left(
  \fracs{1}{2} |\ofd^{(a)}- F_d| 
+ \fracs{1}{2} |\ofd^{(b)}- F_d| 
+  |\ofd - F_d|
\right)^2 \\
&\leq
|D| \sum_d \left(
  |\ofd^{(a)}\!-\! F_d|^2 
+ |\ofd^{(b)}\!-\! F_d|^2 
+ 2 |\ofd     \!-\! F_d|^2
\right).
\end{align*}
For the last term in the summand on the right-most side, we have
\begin{eqnarray*}
\EE \left[ |\ofd - F_d|^2 \right] = \EE \left[ \EE \left[ |\ofd - F_d|^2 \, |\, X \right] \right] = \frac{1}{2^{\ell}}\EE \left[ \VV\left[ f_d(X,Y) \,|\, X\right] \right] .
\end{eqnarray*}
Similarly
\begin{eqnarray*}
\EE \left[ |\ofd^{(a)}- F_d|^2 \right] = \EE \left[ |\ofd^{(b)}- F_d|^2 \right] = \frac{1}{2^{\ell-1}}\EE \left[ \VV\left[ f_d(X,Y) \,|\, X\right] \right] .
\end{eqnarray*}
Hence, $\VV \left[ Z_\ell \right]$ is bounded by
\begin{align*}
\VV \left[ Z_\ell \right] & \leq \EE \left[ \left| Z_\ell \right|^2 \right] \\
& \leq |D|
\sum_d \left(
 \EE\left[ |\ofd^{(a)}- F_d|^2\right] 
+ \EE\left[ |\ofd^{(b)}- F_d|^2 \right] 
+ 2 \EE\left[ |\ofd     - F_d|^2 \right] 
\right). \\
& = \frac{6|D|}{2^\ell}\sum_{d}\EE \left[ \VV [ f_d(X,Y) \,| X ]\right] ,
\end{align*}
which completes the proof.
\end{proof}

The theorem shows that the parameters for the MLMC theorem are $\beta=1$, 
and in view of Remark~\ref{rem:alpha_beta}, $\alpha \geq 1/2$.
Since $\gamma=1$ by the definition of $Z_{\ell}$, 
the MLMC estimator is in the second regime, with $\beta=\gamma$, 
so that the total cost is $O(\eps^{-2} (\log \eps)^2)$.
This compares favourably with the cost of $O(\eps^{-2-1/\alpha})$ 
for the standard Monte Carlo estimator, where the exponent increases up to $4$
in the worst case. In the proof of the theorem, the antithetic property of the estimator, 
i.e., $\fracs{1}{2}(\ofd^{(a)}+\ofd^{(b)}) - \ofd = 0$, is not exploited.
In fact, the same upper bound on the variance can be obtained even for the alternative $\ofd^{(a)}-\ofd$.
In what follows, we prove a stronger result on the variance under somewhat demanding assumptions 
to exploit the antithetic structure of $Z_{\ell}$.

In fact, the MLMC variance can be analysed by following the approach 
used by \citet[][Theorem 5.2]{gs14}.
Define
\[
F_d(X) = \EE \left[f_d(X,Y) | X\right], ~~~~
d_{opt}(X) = \arg\max_{d} F_d(X)
\]
so the domain for $X$ is divided into a number of regions 
in which the optimal decision $d_{opt}(X)$ is unique, with a 
dividing decision manifold $K$ on which 
$d_{opt}(X)$ is not uniquely-defined.

Again note that $\fracs{1}{2}(\ofd^{(a)}+\ofd^{(b)}) - \ofd = 0$,
and therefore $Z_\ell =0$ if the same decision $d$ maximises each 
of the terms in its definition.
When $\ell$ is large and so there 
are many samples, $\ofd^{(a)}, \ofd^{(b)}, \ofd$ will all be close 
to $F_d(X)$, and therefore it is highly likely that 
$Z_\ell =0$ unless $X$ is very close to $K$ at which there
is more than one optimal decision.
This idea leads to an improved theorem on the MLMC variance,
but we first need to make three assumptions.

\begin{assumption}
\label{assp:1}
$\EE\left[ |f_d(X,Y)|^p \right]$ is finite for all $p\geq 2$.\\
\noindent
Comment: this enables us to bound the difference between 
$\ofd^{(a)}, \ofd^{(b)}, \ofd$ and $F_d(X)$.
\end{assumption}
\begin{assumption}
\label{assp:2}
There exists a constant $c_0 > 0$ such that for all $0 <\Eps<1$
\[
\PP\left( \min_{x\in K} \| X - x \| \leq \Eps\right)
\leq c_0 \Eps.
\]
Comment: this bounds the probability of $X$ being close 
to the decision manifold $K$.
\end{assumption}
\begin{assumption}
\label{assp:3}
There exist constants $c_1, c_2>0$ such that if
$X \notin K$,
then
\[
\max_d F_d(X) - \max_{d\neq d_{opt}(X)}  F_d(X) 
 > \min\left( c_1, c_2 \min_{x\in K} \| X - x \| \right).
\]
Comment: on $K$ itself there are at least 2 decisions $d_1, d_2$
which yield the same optimal value $F_d(X)$; this assumption 
ensures at least a linear divergence between the values as 
$X$ moves away from $K$.
\end{assumption}

\begin{theorem}
\label{thm:1}
If Assumptions \ref{assp:1}-\ref{assp:3} are satisfied, and 
$Z_\ell$ is as defined previously for level $\ell$, 
then for any $\delta >0$
\[
\VV \left[ Z_\ell \right] = o(2^{-(3/2-\delta)\ell}), \qquad
\EE \left[ Z_\ell \right] = o(2^{-(1-\delta)\ell}).
\]
\end{theorem}

\noindent
{\it Comment: a similar $O(N^{-3/2})$ convergence rate for the 
variance is proved in Theorem 2.3 in \citep{bhr15} for a different nested simulation application.}

Before going into the detailed proof of the theorem, we give a heuristic explanation on the variance analysis below:
\begin{itemize}
\item Due to Assumption~\ref{assp:1} and Lemma~\ref{lemma:1} shown below, $\ofd-F_d=O(2^{-\ell/2})$;
\item Due to Assumption~\ref{assp:2}, there is an $O(2^{-\ell/2})$ probability of $X$
being within distance $O(2^{-\ell/2})$ from the decision manifold $K$, 
in which case $Z_\ell = O(2^{-\ell/2})$;
\item If $X$ is further away from $K$, Assumption~\ref{assp:3} ensures that
there is a clear separation between different decision values, and hence the antithetic property
of the estimator can be exploited well to give $Z_\ell=0$ with high probability;
\item This results in 
\begin{align*}
\EE[Z_\ell] & = O(2^{-\ell/2}) \times O(2^{-\ell/2}) = O(2^{-\ell}),\\
\EE[Z_\ell^2] & = O(2^{-\ell/2}) \times (O(2^{-\ell/2}))^2 = O(2^{-3\ell/2}),
\end{align*}
so that we have $\alpha\approx 1$ and $\beta\approx 3/2$.
\end{itemize}

To prepare for the proof of the main theorem, we first need a result 
concerning the deviation of an average of $N$ values from 
the expected mean. Suppose $X$ is a real random variable with zero mean, 
and let $\XN$ be an average of $N$ i.i.d.~samples 
$X_n, n= 1, 2, \ldots, N$.  
For $p=2$, we have 
$\EE[\XN^2] = N^{-1} \EE[X^2]$, and hence
$\PP[ |\XN| > c ] \leq \EE[X^2] / (c^2 N)$.
For larger values of $p$ for which $\EE[|X|^p]$ is finite, 
we have the following lemma:

\begin{lemma}
\label{lemma:1}
For $p\geq 2$, if $\EE[|X|^p]$ is finite then there exists a 
constant $C_p$, depending only on $p$, such that
\begin{align*}
\EE[ |\XN|^p] & \leq C_p N^{-p/2} \EE\left[ |X|^p \right], \\
\PP[ |\XN| > c ] & \leq C_p \EE[|X|^p] / (c^2 N)^{p/2}.
\end{align*}
\end{lemma}

\begin{proof}
The discrete Burkholder-Davis-Gundy inequality \citep{bdg72} gives us
\begin{align*}
\EE[ |\XN|^p]
& \leq C_p \EE\left[ \left(N^{-2} \sum_{n=1}^N X_n^2\right)^{p/2} \right] \\
&  \leq C_p \EE\left[ N^{-p/2-1} \sum_{n=1}^N |X_n|^p \right]
 = C_p N^{-p/2} \EE\left[ |X|^p \right],
\end{align*}
where $C_p$ is a constant depending only on $p$.  The second result
follows immediately from the Markov inequality.
\end{proof}

\begin{proof}[Proof of Theorem~\ref{thm:1}]
The analysis follows the approach used by \citet{ghm09} and
\citet[][Theorem 5.2]{gs14}.

For a particular value of $\delta$, we define
$\Eps = 2^{-(1/2-\delta/2)\ell}$, and consider the events
\begin{align*}
A &\equiv \left\{\min_{x\in K}\| X - x \| \leq \Eps \right\}, \\
B &\equiv \displaystyle \bigcup_d \left\{ 
\max\left( |\ofd^{(a)} - F_d |, |\ofd^{(b)} - F_d |, |\ofd - F_d | \right)
\geq \fracs{1}{2}c_2 \Eps \right\},
\end{align*}
where $c_2$ is as defined in Assumption \ref{assp:3}.

Using 
$\one_{A}$ to indicate the indicator function for event $A$, and
$A^c$ to denote the complement of $A$, we have
\begin{align}
& \EE\left[ \left(
\fracs{1}{2} (\max_d \ofd^{(a)} + \max_d \ofd^{(b)}) - \max_d \ofd
\right)^{2} \right]  \nonumber \\
& = 
\EE\left[ \left(
\fracs{1}{2} (\max_d \ofd^{(a)} + \max_d \ofd^{(b)}) - \max_d \ofd
\right)^{2} \one_{A \cup B} \right] \nonumber \\
& \quad +  
 \EE\left[ \left(
\fracs{1}{2} (\max_d \ofd^{(a)} + \max_d \ofd^{(b)}) - \max_d \ofd
\right)^{2} \one_{A^c \cap B^c} \right].
\label{eq:split}
\end{align}

Looking at the first of the two terms on the r.h.s.~of (\ref{eq:split}), 
then H{\"o}lder's inequality gives
\begin{align*}
& \EE\left[ \left(
\fracs{1}{2} (\max_d \ofd^{(a)} + \max_d \ofd^{(b)}) - \max_d \ofd
\right)^{2} \one_{A \cup B} \right]
 \\
& \leq
\EE\left[ \left(
\fracs{1}{2} (\max_d \ofd^{(a)} + \max_d \ofd^{(b)}) - \max_d \ofd
\right)^{2p} \right]^{1/p} 
\left( \PP(A) + \PP(B) \right)^{1/q}
\end{align*}
for any $p,q\geq 1$, with $p^{-1}+q^{-1}=1$.

Now, $\PP(A) \leq c_0\, \Eps$ due to Assumption \ref{assp:1}, 
and
\[
\PP(B) \leq \sum_d \left(
  \PP( |\ofd^{(a)}-F_d| \geq \fracs{1}{2}\Eps)
+ \PP( |\ofd^{(b)}-F_d| \geq \fracs{1}{2}\Eps)
+ \PP( |\ofd-F_d| \geq \fracs{1}{2}\Eps)
\right).
\]
Due to Lemma \ref{lemma:1},
\[
\PP(\, |\ofd -F_d| \geq \fracs{1}{2}\Eps \, | \, X)
= \EE\left[ {\bf 1}_{|\ofd -F_d| \geq \fracs{1}{2}\Eps} \, | \, X\right]
\leq C_m  \EE\left[ |f_d-F_d|^m \, | \, X \right] / (\Eps^2 2^\ell)^{m/2},
\]
for any $m\geq 2$. Taking an outer expectation with respect to $X$, 
the tower property then gives
\[
\PP( |\ofd  -F_d| \geq \fracs{1}{2}\Eps )
= \EE\left[ {\bf 1}_{|\ofd -F_d| \geq \fracs{1}{2}\Eps} \right]
\leq C_m \EE\left[ |f_d-F_d|^m \right] / (\Eps^2 2^\ell)^{m/2}.
\]

Similar bounds exists for 
$\PP(|\ofd^{(a)}-F_d| \geq \fracs{1}{2}\Eps)$ and
$\PP(|\ofd^{(b)}-F_d| \geq \fracs{1}{2}\Eps)$.
We can take $m$ to be sufficiently large so that
$
\fracs{1}{2} m - \fracs{1-\delta}{2} m > \fracs{1-\delta}{2}
$
and hence $\PP(B)= o(2^{-(1-\delta)\ell/2})$. Then, $q$ can 
be chosen sufficiently close to 1 so that
$
\left( \PP(A) + \PP(B) \right)^{1/q} \ =\ o(2^{-(1/2-\delta)\ell}).
$

Applying Jensen's inequality to (\ref{eq:max_to_sum}) twice, we obtain
\begin{align*}
& \left( \fracs{1}{2} (\max_d \ofd^{(a)} + \max_d \ofd^{(b)}) - \max_d \ofd \right)^{2p} \\
&\leq
|D|^{2p-1}
\sum_d \left(
  \fracs{1}{2} |\ofd^{(a)}- F_d| 
+ \fracs{1}{2} |\ofd^{(b)}- F_d| 
+ |\ofd- F_d|
\right)^{2p} \\
&\leq
(2|D|)^{2p-1}
\sum_d \left(
  \fracs{1}{2} |\ofd^{(a)}- F_d|^{2p} 
+ \fracs{1}{2} |\ofd^{(b)}- F_d|^{2p} 
+              |\ofd - F_d|^{2p}
\right).
\end{align*}
It follows from Lemma \ref{lemma:1} that
\begin{align*}
\EE[\,|\ofd - F_d|^{2p}] 
& =  \EE\left[ \EE[|\ofd - F_d|^{2p} \, | \, X] \right] \\
&\leq C_{2p} 2^{-p \ell} \EE\left[ \EE[ |f_d(X,Y) - \EE[f_d(X,Y)\,|\, X]|^{2p} \, | \, X] \right]\\
&=    C_{2p} 2^{-p \ell} \EE[ |f_d(X,Y) - \EE[f_d(X,Y)]|^{2p}],
\end{align*}
so Assumption \ref{assp:1} implies that
$\EE[|\ofd - F_d|^{2p}] = O(2^{-p \ell})$, with similar 
bounds for $\ofd^{(a)}$ and $\ofd^{(b)}$.  Hence,
\[
\EE\left[ \left(
\fracs{1}{2} (\max_d \ofd^{(a)} + \max_d \ofd^{(b)}) - \max_d \ofd
\right)^{2p} \right]^{1/p} 
= O(2^{-\ell}),
\]
and therefore the first term on the r.h.s.~of (\ref{eq:split})
has bound $o(2^{-(3/2-\delta)\ell})$.

We now consider the second term on the r.h.s.~of (\ref{eq:split}).
For any sample in $A^c\cap B^c$, we have
$\displaystyle
\min_{x\in K}\| X - x \| \geq \Eps,
$
and
$
|\ofd^{(a)} - F_d | < \fracs{1}{2}c_2 \Eps , ~
|\ofd^{(b)} - F_d | < \fracs{1}{2}c_2 \Eps , ~
|\ofd        - F_d | < \fracs{1}{2}c_2 \Eps,
$
for all $d$.
For a particular outer sample $X$, if 
$d \neq d_{opt}(X)$ then using Assumption \ref{assp:3} we have
\begin{align*}
\ofd_{opt} - \ofd & = ({F_d}_{opt} - F_d) + (\ofd_{opt}-{F_d}_{opt}) - (\ofd-F_d) \\
             & > \min(c_1, c_2 \Eps) - \fracs{1}{2} c_2 \Eps  - \fracs{1}{2} c_2 \Eps 
 = \min(c_1 - c_2 \Eps,  0)
\end{align*}
If $\ell$ is sufficiently large so that $c_2 \Eps < c_1$, then $\ofd_{opt} - \ofd > 0$ 
and hence $d_{opt} = \arg \max_d \ofd$.  The same argument applies to 
$\ofd_{opt}^{(a)} - \ofd^{(a)}$ and $\ofd_{opt}^{(b)} - \ofd^{(b)}$, so the conclusion is 
that in all three cases, $d_{opt}$ is the decision which maximises 
$\ofd^{(a)}$, $\ofd^{(b)}$ and $\ofd$, and therefore
\[
\fracs{1}{2} (\max_d \ofd^{(a)} + \max_d \ofd^{(b)}) - \max_d \ofd
=\fracs{1}{2} ( \ofd_{opt}^{(a)} +  \ofd_{opt}^{(b)}) - \ofd_{opt}
=0.
\]
Hence, for sufficiently large $\ell$, the second term is zero,
which concludes the proof for the bound on $\VV[Z_\ell]$ and
the bound on $\EE[Z_\ell]$ is obtained similarly.
\end{proof}

The conclusion from the theorem is that the parameters for the MLMC 
theorem are $\beta \!\approx\! 3/2$, $\alpha\!\approx\! 1$, and 
$\gamma \!=\! 1$, giving the optimal complexity of $O(\eps^{-2})$.
Again, this compares favourably with the cost of $O(\eps^{-3})$ for the standard Monte Carlo estimator.

\section{Numerical results}\label{sec:numerics}

\subsection{Simple test cases}

To validate the importance of the assumptions made in the variance analysis,
several simple examples are tested here.
Let $X$ and $Y$ be independent univariate standard normal random variables, 
and let us consider two-treatment decision problems with $f_1(X,Y)=0$ and either
\begin{enumerate}
\item $f_2(X,Y)=X+Y$,
\item $f_2(X,Y)=X^3+Y$, or
\item $f_2(X,Y)=\begin{cases}
X+Y+1, & X<-1, \\
Y, & -1\leq X\leq 1, \\
X+Y-1, & X>1.
\end{cases} $
\end{enumerate}
It is easy to check that this simple test case with the first choice of $f_2$ satisfies 
all of Assumptions \ref{assp:1}-\ref{assp:3}, while the other cases with the second and 
third choices of $f_2$ do not.
With the second choice of $f_2$, we have $F_1(X)=0$ and $F_2(X)=X^3$, so that $K=\{0\} \subset \RR$ and 
\[
\max_d F_d(X) - \max_{d\neq d_{opt}(X)}  F_d(X) \ = \ |X|^3,
\]
which implies that there exist no constants $c_1,c_2>0$ such that Assumption \ref{assp:3} is satisfied.
For the third choice of $f_2$, we have $K=[-1, 1]\subset \RR$ whose probability measure is not zero. 
Hence, by considering the limiting situation $\Eps\to 0$ in Assumption \ref{assp:2}, 
we see that there exists no constant $c_0>0$ such that Assumption \ref{assp:2} is satisfied.

The results for the first choice of $f_2$ are shown in Figure \ref{fig:test1}.
The left top plot shows the behaviours of the variances of both $P_{\ell}$ and $Z_{\ell}$, 
where the variances are estimated by using $N=2\times 10^5$ random samples at each level.
Note that the logarithm of the empirical variance in base $2$ versus the level is plotted here.
The slope of the line for $Z_{\ell}$ is $-1.43$, indicating that $\VV[Z_{\ell}]=O(2^{-1.43 \ell})$.
This result is in good agreement with Theorem \ref{thm:1} which holds for decision models 
satisfying Assumptions \ref{assp:1}-\ref{assp:3}.

\begin{figure}
\centering
\includegraphics[width=\textwidth]{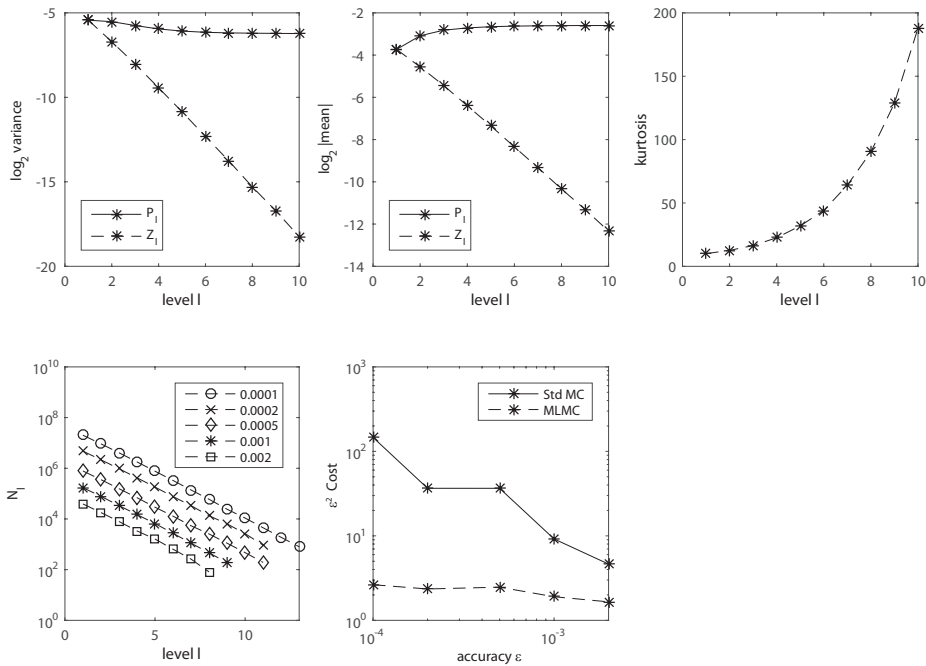}
\caption{MLMC results for simple test case with the first choice of $f_2$.}
\label{fig:test1}
\end{figure}

The middle top plot shows the behaviours of the estimated mean values of both $P_{\ell}$ and $Z_{\ell}$.
The slope of the line for $Z_{\ell}$ is approximately $-1$, which implies that $\EE[Z_{\ell}]=O(2^{- \ell})$.
This is again in good agreement with Theorem \ref{thm:1}.

The right top plot shows the behaviour of the estimated kurtosis of $Z_{\ell}$.
The way in which the kurtosis increases with the level also confirms that 
the MLMC corrections are increasingly dominated by a few rare samples 
yielding $Z_{\ell}\neq 0$, corresponding to outer samples $X$ which are close 
to the decision manifold $K$ across which the optimal decision $d_{opt}$ changes.

Using the implementation due to \citet[][Algorithm 1]{giles15}, the maximum level $L$ and 
the computational costs $N_{\ell}$ for levels $\ell=1,\ldots,L$, required for the combined multilevel estimator 
to achieve an MSE less than $\eps^2$, are estimated.
Each line in the left bottom plot shows the values of $N_{\ell}$, $\ell=1,\ldots,L$, for a particular value of $\eps$.
As expected, the number of samples varies with the level such that
many more samples are allocated on the coarsest levels,
which is in good agreement with the optimal allocation of
computational effort given by 
$N_\ell\propto \eps^{-2}\sqrt{V_\ell/C_\ell}\propto \eps^{-2} 2^{-(\beta+\gamma)\ell/2}$ \citep{giles15}.
It is also shown here that, as the value of $\eps$ decreases, the maximum level $L$ increases to 
ensure the weak convergence $|\EE\left[P-P_L \right]|\leq \eps/\sqrt{2}$.

The middle bottom plot shows the behaviour of the total computational cost
\[
C = \sum_{\ell=1}^{L}2^{\ell}N_{\ell},
\]
to achieve an MSE less than $\eps^2$.
Since it is expected from the MLMC theorem that $\eps^2 C$ is independent of $\eps$, we plot $\eps^2 C$ 
versus $\eps$ here.
Indeed, it can be seen that $\eps^2 C$ is only slightly dependent on $\eps$, indicating that the MLMC estimator 
gives the optimal complexity of $O(\eps^{-2})$.
This result compares favourably with the result for the standard (in this case, nested) Monte Carlo method.
The superiority of the MLMC method becomes more evident as the desired accuracy $\eps$ decreases.
For instance, for $\eps=10^{-4}$, the MLMC method is more than 50 times more efficient.

Let us move on to the second and third choices of $f_2$.
Since these test cases do not satisfy one of Assumptions \ref{assp:1}-\ref{assp:3}, Theorem \ref{thm:1} 
does not apply and it is expected from Theorem \ref{thm:nearly-optimal} that the MLMC estimator 
achieves the nearly optimal complexity of $O(\eps^{-2} (\log \eps)^2)$.
The results for the second and third choices of $f_2$ are shown in Figures \ref{fig:test2} and \ref{fig:test3}, respectively.

For the second choice of $f_2$, it is seen from the first two top plots that the slopes of the lines 
for the variance and the mean value of $Z_{\ell}$ are $-1.12$ and $-0.64$, respectively, 
which are slightly better than the values $-1$ and $-0.5$ which are to be expected from the theory.
In the right top plot, the kurtosis increases with the level but not so significantly as compared to the first test case.
Because of a smaller value of $\alpha$, we can observe in the left bottom plot 
that the maximum level to ensure the weak convergence becomes large.
Still, the superiority of the MLMC method over the standard Monte Carlo method is prominent.
For $\eps=10^{-4}$, the MLMC method is approximately 3000 times more efficient.
Similar results are also obtained for the third choice of $f_2$.

\begin{figure}
\centering
\includegraphics[width=\textwidth]{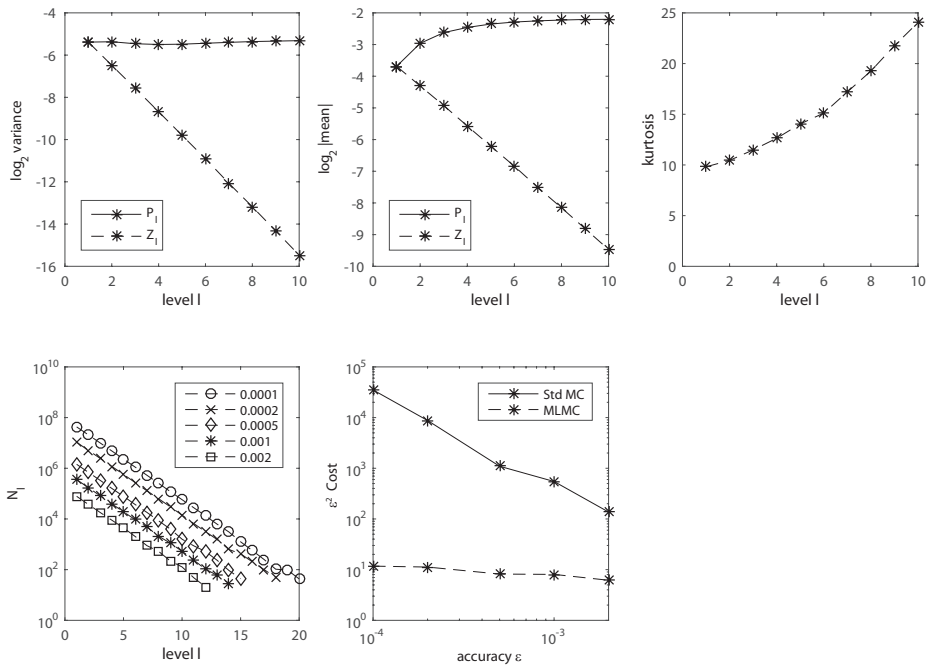}
\caption{MLMC results for simple test case with the second choice of $f_2$.}
\label{fig:test2}
\end{figure}

\begin{figure}
\centering
\includegraphics[width=\textwidth]{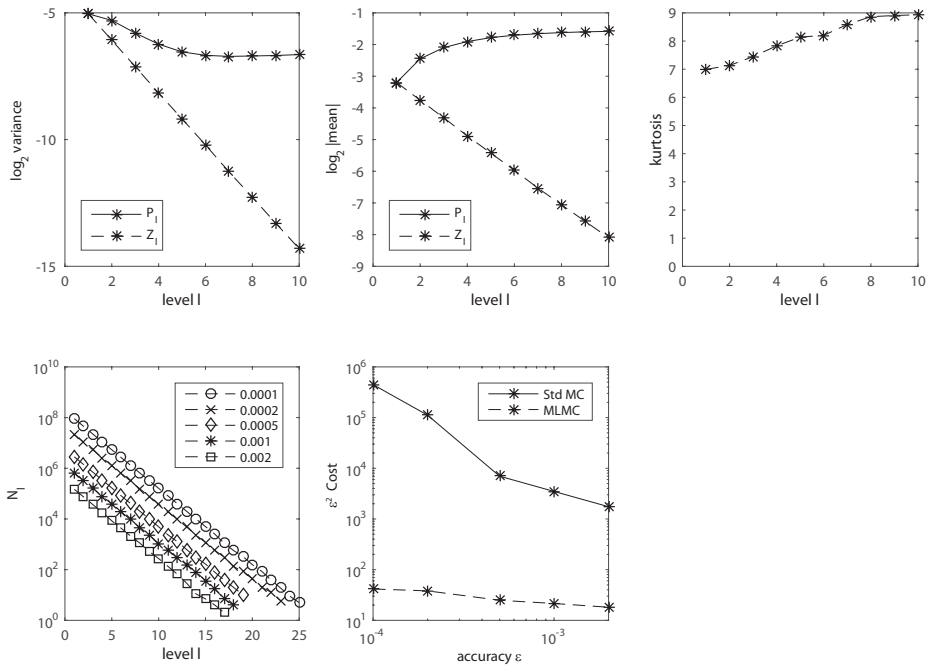}
\caption{MLMC results for simple test case with the third choice of $f_2$.}
\label{fig:test3}
\end{figure}

\subsection{Medical decision model}

To demonstrate the practical usefulness of the MLMC estimator, 
the medical decision model introduced in \citet{bkoc07} is tested.
Let $X\cup Y=(X_1,\ldots,X_{19})$ with each univariate random variable $X_j$ 
following the normal distribution with mean $\mu_j$ and standard deviation $\sigma_j$ independently 
except that $X_5,X_7,X_{14},X_{16}$ are pairwise correlated with a correlation coefficient $\rho=0.6$.
The values for $\mu_j$ and $\sigma_j$ and the medical meaning of $X_j$ are listed in Table \ref{table:1}.
The problem to be tested is a two-treatment decision problem with 
\begin{eqnarray*}
f_1(X,Y) & = & \lambda \left( X_5 X_6 X_7 + X_8 X_9 X_{10}\right) - (X_1 + X_2 X_3 X_4), \ \text{and} \\
f_2(X,Y) & = & \lambda \left( X_{14} X_{15} X_{16} + X_{17} X_{18} X_{19}\right) - (X_{11} + X_{12} X_{13} X_4),
\end{eqnarray*}
where $\lambda$ denotes the monetary valuation of health and is set to $10^4$ ($\pounds$).
In what follows, we call this decision model the BKOC test case, named after the authors of \citet{bkoc07}.

\begin{table}
  \caption{Variables in the BKOC test case (Table 2 in \cite{bkoc07}).}
  \label{table:1}
  \centering
  \begin{tabular}{cccc}
    \hline
    variable & $\mu_j$  & $\sigma_j$ & meaning \\
    \hline 
    $X_1$ & 1000 & 1 & Cost of drug ($\pounds$) \\
    $X_2$ & 0.1 & 0.02 & Probability of admissions \\
    $X_3$ & 5.2 & 1.0 & Days in hospital \\
    $X_4$ & 400 & 200 & Cost per day ($\pounds$) \\
    $X_5$ & 0.7 & 0.1 & Probability of responding \\
    $X_6$ & 0.3 & 0.1 & Utility change if response \\
    $X_7$ & 3.0 & 0.5 & Duration of response (years) \\
    $X_8$ & 0.25 & 0.1 & Probability of side effects \\
    $X_9$ & -0.1 & 0.02 & Change in utility if side effect \\
    $X_{10}$ & 0.5 & 0.2 & Duration of side effect (years) \\
    $X_{11}$ & 1500 & 1 & Cost of drug  ($\pounds$) \\
    $X_{12}$ & 0.08 & 0.02 & Probability of admissions \\
    $X_{13}$ & 6.1 & 1.0 & Days in hospital \\
    $X_{14}$ & 0.8 & 0.1 & Probability of responding \\
    $X_{15}$ & 0.3 & 0.05 & Utility change if response \\
    $X_{16}$ & 3.0 & 1.0 & Duration of response (years) \\
    $X_{17}$ & 0.2 & 0.05 & Probability of side effects \\
    $X_{18}$ & -0.1 & 0.02 & Change in utility if side effect \\
    $X_{19}$ & 0.5 & 0.2 & Duration of side effect (years) \\
    \hline
  \end{tabular}
\end{table}

The results for the BKOC test case with $X=(X_5,X_{14})$ are shown in Figure \ref{fig:bkoc1}.
From the first two top plots we see that the slopes of the lines 
for the variance and the mean value of $Z_{\ell}$ are $-1.352$ and $-0.89$, respectively,
indicating that the MLMC estimator is in the first regime, with $\beta>\gamma$.
The behaviour of the kurtosis of $Z_{\ell}$, shown in the right top plot, is quite similar 
to that observed for the simple test case with the first choice of $f_2$.
As expected, most of the computational cost is actually incurred on the coarsest levels, and 
the MLMC method gives savings of factor more than 100 as compared to the standard Monte Carlo method 
for the desired accuracy $\eps = 0.1$.

As shown in Figure \ref{fig:bkoc2} and \ref{fig:bkoc3}, respectively,
both of the results for the BKOC test case with $X=(X_5,X_6,X_{14},X_{15})$ and $X=(X_7,X_{16})$ are quite similar to
the case with $X=(X_5,X_{14})$, and the MLMC method gives savings of factor up to 100.

In order to achieve an MSE less than 1, the MLMC method needs the total computational costs of 
$C=4.1\times 10^7, 3.0\times 10^7, 2.2\times 10^7$ for the three respective cases, 
giving the estimates of the difference $\mbox{EVPI} - \mbox{EVPPI}$ as 799, 206, and 509.
The total computational costs for the standard Monte Carlo method are found to be approximately 10 times larger for all cases.
The standard Monte Carlo method using $10^7$ random samples of $(X,Y)$ yields the estimate of EVPI as 1047.
Thus, the EVPPI values for the three cases are estimated as 248, 841 and 538.

\begin{figure}
\centering
\includegraphics[width=\textwidth]{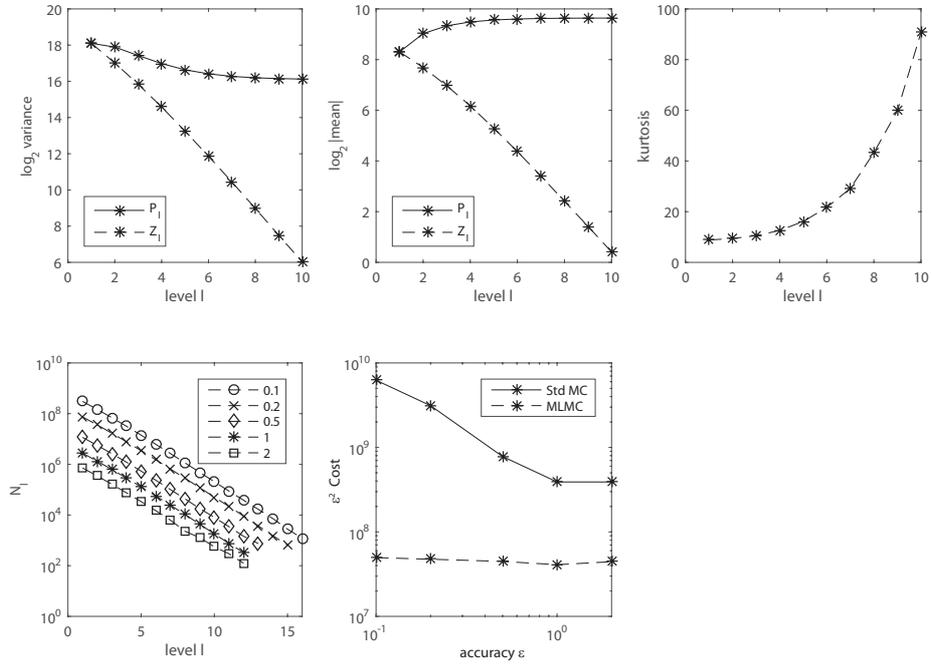}
\caption{MLMC results for the BKOC test case with $X=(X_5,X_{14})$.}
\label{fig:bkoc1}
\end{figure}

\begin{figure}
\centering
\includegraphics[width=\textwidth]{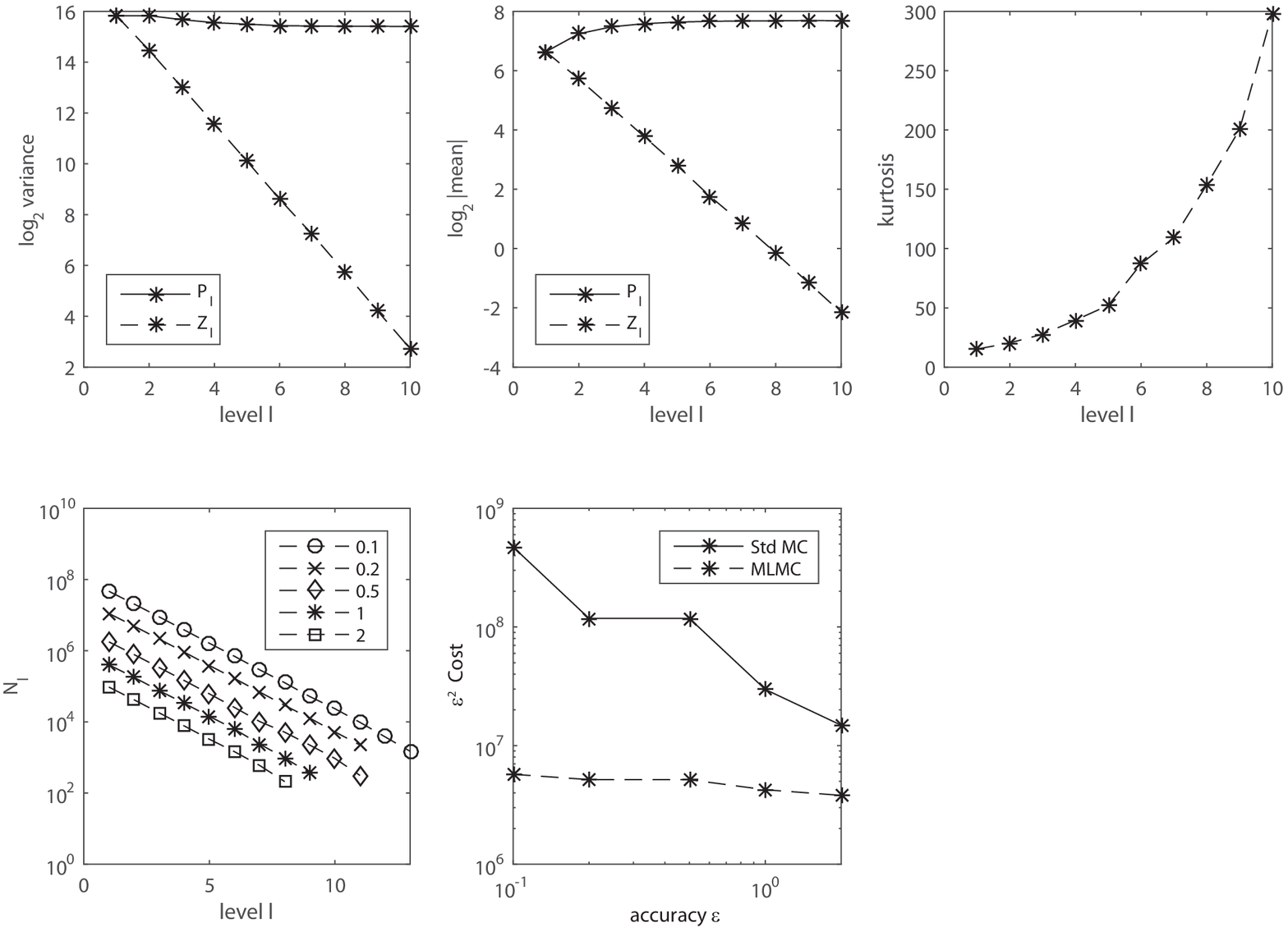}
\caption{MLMC results for the BKOC test case with $X=(X_5,X_6,X_{14},X_{15})$.}
\label{fig:bkoc2}
\end{figure}

\begin{figure}
\centering
\includegraphics[width=\textwidth]{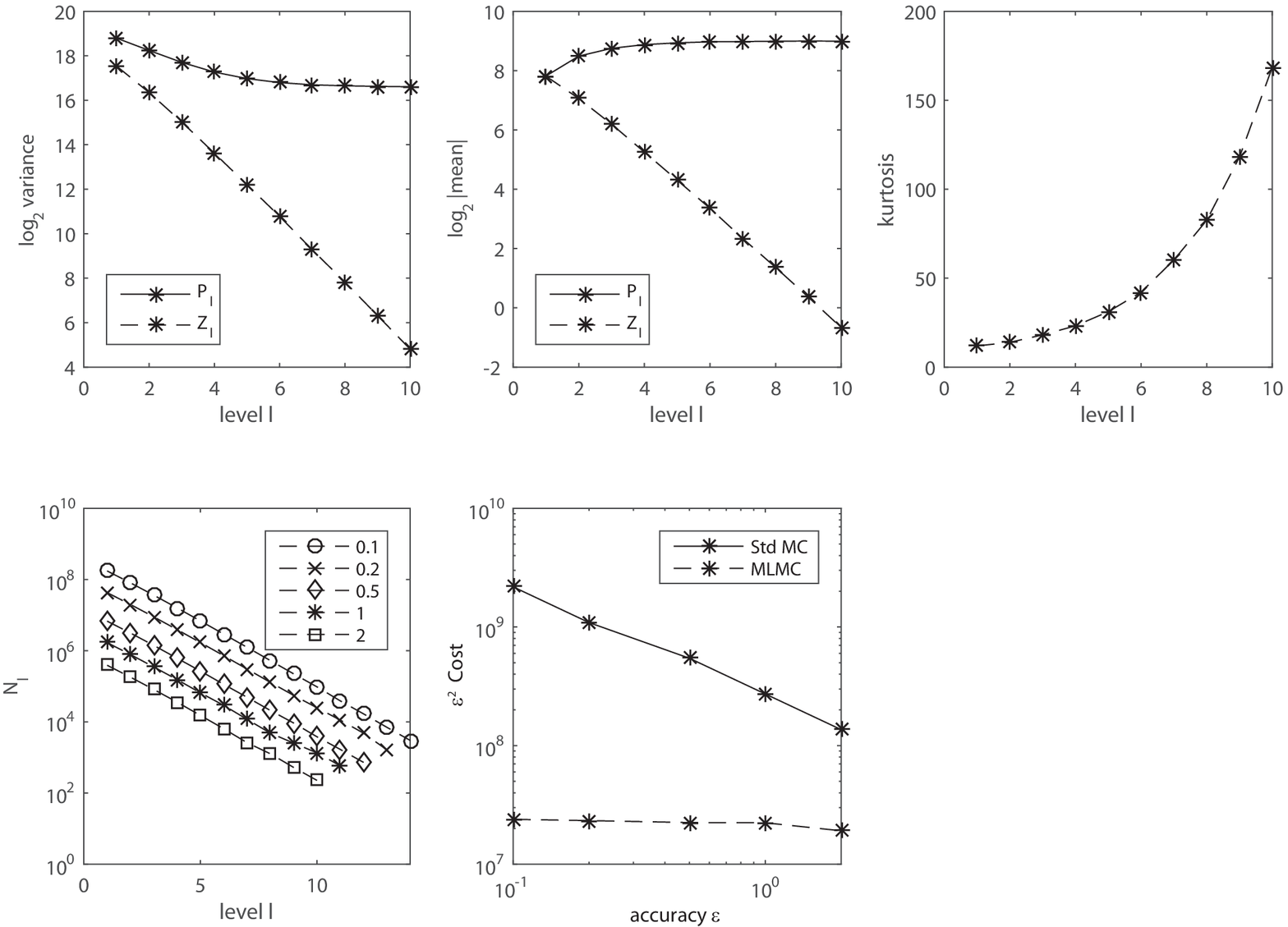}
\caption{MLMC results for the BKOC test case with $X=(X_7,X_{16})$.}
\label{fig:bkoc3}
\end{figure}

\section{Conclusions}

In this paper we have developed a Multilevel Monte Carlo method for the 
estimation of the expected value of partial perfect information, EVPPI,
which is one of the most demanding nested expectation applications.
The essential difficulty in the theoretical analysis lies in how to deal with 
the maximum of an unconditional expectation. We provide a set of assumptions 
on a decision model to exploit the antithetic property of the estimator, and then
numerical analysis proves that a root-mean-square accuracy of $\eps$ 
can be achieved at a computational cost which is $O(\eps^{-2})$, and 
this is also supported by numerical experiments. As we already announced in
\citep{ggtfw17}, the MLMC estimator introduced in this paper works
quite well for \emph{real} medical application which measures the cost-effectiveness of
novel oral anticoagulants in atrial fibrillation. 
The details on this application shall be summarised in the near future.

Future research will address the following topics:
\begin{itemize}
\item
an extension to handle input distributions which are defined empirically,
such as through the use of MCMC methods to sample from a Bayesian posterior 
distribution;

\item
the use of quasi-random numbers in place of pseudo-random numbers, 
which leads to the Multilevel Quasi-Monte Carlo method which is 
capable of additional substantial savings \citep{gw09};

\item
the use of an adaptive number of inner samples, following the ideas 
of \cite{bdm11}, since it is only the outer samples which are near the 
decision manifold $K$ which require great accuracy for the inner 
conditional expectation.
\end{itemize}

\section*{Acknowledgements}
The authors would like to thank Dr.~Howard Thom of the University of Bristol for useful discussions and comments.
The research of T. Goda was supported by JSPS Grant-in-Aid for Young Scientists (No.~15K20964) 
and Arai Science and Technology Foundation.

\end{document}